\documentclass{amsart}

\usepackage{amsmath}
\usepackage{amsthm}
\usepackage{amssymb}
\usepackage{tikz, tikz-cd}
\usepackage{ytableau}
\usepackage{outlines}
\usepackage{hyperref}

\newcommand{\sg}{\sigma}

\def\Des{\operatorname{Des}}
\def\asc{\operatorname{asc}}

\def\inv{\operatorname{inv}}
\def\Inv{\operatorname{Inv}}
\def\maj{\operatorname{maj}}

\def\Coinv{\operatorname{Coinv}}

\def\S{\mathfrak{S}}

\def\multiset#1#2{\ensuremath{\left(\kern-.3em\left(\genfrac{}{}{0pt}{}{#1}{#2}\right)\kern-.3em\right)}}

\def\wt {\operatorname{wt}}

\def\shape{\operatorname{shape}}

\def\llt{LLT}

\def\arm{\operatorname{arm}}
\def\leg{\operatorname{leg}}

\def\down{\operatorname{down}}
\def\hook{\operatorname{hook}}
\def\weight{\operatorname{wt}}
\def\itab{\operatorname{IFT}}
\def\armset{\operatorname{Arm}}
\def\chromatic{X}

\newcommand{\RS}{\operatorname{RS}}
\newcommand{\CS}{\operatorname{CS}}
\newcommand{\type}{\operatorname{type}}

\newtheorem{prop}{Proposition}[subsection]
\newtheorem{cor}{Corollary}[subsection]
\newtheorem{thm}{Theorem}[subsection]

\newtheorem{defn}{Definition}[subsection]

\title{Macdonald polynomials and chromatic quasisymmetric functions}

\author{J.\ Haglund}
\address{Department of Mathematics, University of Pennsylvania \\ Philadelphia, PA 19104}
\email{\texttt{jhaglund@math.upenn.edu}}
\thanks{The first author was partially supported by NSF grant DMS-1200296.}

\author{A.\ T.\ Wilson}
\address{Department of Mathematics, University of Pennsylvania \\ Philadelphia, PA 19104}
\email{\texttt{andwils@math.upenn.edu}}
\thanks{The second author was partially supported by an NSF Mathematical Sciences Postdoctoral Research Fellowship. }

\begin{document}
\maketitle

\begin{abstract}
We express the integral form Macdonald polynomials as a weighted sum of Shareshian and Wachs' chromatic quasisymmetric functions of certain graphs. Then we use known expansions of these chromatic quasisymmetric functions into Schur and power sum symmetric functions to provide Schur and power sum formulas for the integral form Macdonald polynomials. Since the (integral form) Jack polynomials are a specialization of integral form Macdonald polynomials, we obtain analogous formulas for Jack polynomials as corollaries.
\end{abstract}

\tableofcontents

\section{Introduction}

The \emph{Macdonald polynomials} are a basis $\{P_{\mu}(x;q,t) : \text{partitions } \mu \}$ for the ring of symmetric functions which are defined by certain triangularity relations \cite{macdonald}. This basis has the additional property that it reduces to classical bases, such as the Schur functions and the monomial symmetric functions, after certain specializations of $q$ and $t$.

A second basis $\{J_{\mu}(x;q,t): \text{partitions } \mu \}$, known as the \emph{integral form Macdonald polynomials}, can be obtained by multiplying each $P_{\mu}(x;q,t)$ by a particular scalar in $\mathbb{Q}(q,t)$. These polynomials are known as integral forms because, when expanded into a suitably modified Schur function basis, the resulting coefficients are in $\mathbb{N}[q,t]$ \cite{haiman-positivity}. There is currently no formula for computing these resulting coefficients. The only result in this direction is a combinatorial formula for expanding integral form Macdonald polynomials into monomial symmetric functions \cite{hhl-nsym}.

In \cite{shareshian-wachs}, Shareshian and Wachs define the \emph{chromatic quasisymmetric function} $\chromatic_G(x;t)$ of a graph $G$, which is a $t$-analogue of Stanley's chromatic symmetric function \cite{stanley-chromatic}. Shareshian and Wachs prove that $\chromatic_G(x;t)$ is symmetric for a certain class of graphs $G$; in particular, the graphs in this paper will all have symmetric $\chromatic_G(x;t)$. These $\chromatic_G(x;t)$ also have known formulas for their Schur \cite{shareshian-wachs} and power sum \cite{ath-power} expansions.

Our main result is an expression for the integral form Macdonald polynomials in terms of chromatic quasisymmetric functions of certain graphs. In particular, given a partition $\mu$ we define two graphs: the \emph{attacking graph} of $\mu$, denoted $G_{\mu}$, and the \emph{augmented attacking graph} of $\mu$, denoted $G_{\mu}^{+}$. Both of these graphs have vertex set $\{1,2,\ldots,|\mu|\}$ and every edge of $G_{\mu}$ is an edge of $G_{\mu}^{+}$, which we write as $G_{\mu} \subseteq G_{\mu}^{+}$.  Our main formula, which appears later as Theorem \ref{thm:integral-form-chromatic}, is
\begin{align}
J_{\mu^{\prime}}(x;q,t) &= t^{-n(\mu^{\prime}) + \binom{\mu_1}{2}} (1-t)^{\mu_1} \sum_{G_{\mu} \subseteq H \subseteq G_{\mu}^{+}} \chromatic_{H}(x;t) \\
&\times \prod_{\{u,v\} \in H \setminus G_{\mu}} - \left(1 - q^{\leg_{\mu}(u) + 1} t^{\arm_{\mu}(u)}\right) \nonumber \\
&\times \prod_{\{u,v\} \in G_{\mu}^{+} \setminus H} \left(1 - q^{\leg_{\mu}(u)+1}t^{\arm_{\mu}(u)+1}\right) . \nonumber
\end{align}
where the precise definitions are given in Section \ref{sec:background}.

We use this formula to obtain Schur and power sum expansions of integral form Macdonald polynomials. The coefficients in these expansions are not generally positive, so our formulas cannot be positive. Furthermore, our expansions do have some cancelation in general. We hope that this cancelation can be simplified, or even eliminated, in future work. These expansions, along with our main formula, are the focus of Section \ref{sec:integral-form}.

Finally, the \emph{Jack polynomials} are another basis for the ring of symmetric functions which are obtained by taking a certain limit of integral form Macdonald polynomials. In Section \ref{sec:jack}, we show how to use our previous identities from Section \ref{sec:integral-form} to obtain analogous results for Jack polynomials. The resulting identities are much simpler in this case, and a reader new to this subject may prefer to read Section \ref{sec:jack} before Section \ref{sec:integral-form}.

\section{Background}
\label{sec:background}

\subsection{Integral form Macdonald polynomials}
The \emph{integral form Macdonald polynomials} $J_{\mu}(x;q,t)$ are symmetric functions in variables $x_1, x_2, \ldots$ with coefficients in the field $\mathbb{Q}(q,t)$; we refer the reader to \cite{haglund-book} for a complete definition. In \cite{hhl-nsym}, Haglund, Haiman, and Loehr prove a combinatorial formula for $J_{\mu}(x;q,t)$. Their result can also be thought of as an expansion of $J_{\mu}(x;q,t)$ into monomial symmetric functions. We state Haglund, Haiman, and Loehr's formula here, which we will use as our ``definition'' for $J_{\mu}(x;q,t)$. 

Given a partition $\mu$, let $\mu^{\prime}$ be the conjugate partition of $\mu$. A \emph{filling} of $\mu$ is a map from the cells of the (French) Ferrers diagram of $\mu$ to $\mathbb{Z}_{>0}$. Two cells in a Ferrers diagram \emph{attack} if they are in the same row or if they are one row apart and the right cell is one row above the left cell. A filling is \emph{non-attacking} if the entries in attacking cells are never equal. 

\begin{defn}
We make the following definitions for a cell $u$ in $\mu$.
\begin{itemize}
\item The \emph{arm} of $u$ in $\mu$, denoted $\arm_{\mu}(u)$, is the number of cells strictly to the right of $u$ in its row in $\mu$.
\item The \emph{leg} of $u$ in $\mu$, denoted $\leg_{\mu}(u)$, is the number of cells strictly above $u$ in its row in $\mu$, 
\item $\down_{\mu}(u)$ is the cell immediately below $u$ in $\mu$ (if such a cell exists).
\end{itemize}
\end{defn}

We also use certain statistics $\maj$ and $\inv$ on fillings of $\mu$. We will not define these statistics precisely now; for definitions, see the proof of Theorem \ref{thm:integral-form-chromatic} or the original source \cite{hhl-nsym}.

\begin{figure}
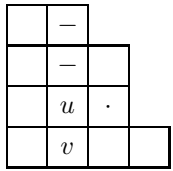

\begin{displaymath}
\begin{ytableau}
\phantom{\cdot}& - \\
\phantom{\cdot}& - &\phantom{cdot} \\
\phantom{\cdot} 	& u	& \cdot  \\
\phantom{\cdot}	&  v	&  \phantom{\cdot}	&  \phantom{\cdot}
\end{ytableau}
\end{displaymath}
\caption{This is the Ferrers diagram of the partition $(4,3,3,2)$. The cell $u$ has $\arm_{\mu}(u) = 1$ (count the dots), $\leg_{\mu}(u) = 2$ (count the dashes), and $\down_{\mu}(u) = v$.}
\label{fig:arm-leg}
\end{figure}

\begin{thm}[\cite{hhl-nsym}]
\label{thm:hhl-nsym}
\begin{align}
J_{\mu^{\prime}}(x;q,t) &= \sum_{\substack{\sg: \mu \to \mathbb{Z}_{>0} \\ \sg \text{ non-attacking}}} x^{\sg} q^{\maj(\sg, \mu)} t^{n(\mu^{\prime}) - \inv(\sg, \mu)} \\
&\times \prod_{\substack{u \in \mu \\ \sg(u) = \sg(\down_{\mu}(u))}} \left( 1 - q^{\leg_{\mu}(u)+1} t^{\arm_{\mu}(u)+1} \right) \nonumber \\
&\times \prod_{\substack{u \in \mu \\ \sg(u) \neq \sg(\down_{\mu}(u))}} \left( 1 - t \right) . \nonumber
\end{align}
where each cell in the bottom row of $\mu$ contributes to the last product.
\end{thm}

\subsection{Jack polynomials}
The \emph{Jack polynomials} $J_{\mu}^{(\alpha)}(x)$ are symmetric functions in the variables $x_1, x_2, \ldots$ with an extra parameter $\alpha$. They can be obtained by taking a certain limit of integral form Macdonald polynomials. 

\begin{defn}
For any partition $\mu \vdash n$, the (integral form) Jack polynomial is 
\begin{align}
J_{\mu}^{(\alpha)}(x) &= \lim_{t \to 1} \frac{J_{\mu} \left(x;t^{\alpha},t \right)}{(1-t)^n} .
\end{align}
\end{defn} 

Setting $\alpha = 1$ yields a scalar multiple of the Schur function $s_{\mu}$. In \cite{knop-sahi}, Knop and Sahi prove a combinatorial formula for $J_{\mu}^{(\alpha)}(x)$, which we recount here. 

\begin{defn}
The \emph{$\alpha$-hook} of $u$ in $\mu$ is 
\begin{align}
\hook^{(\alpha)}_{\mu}(u) = \alpha(\leg_{\mu}(u)+1) + \arm_{\mu}(u).
\end{align}
\end{defn}

\begin{thm}{\cite{knop-sahi}}
\begin{align}
J_{\mu^{\prime}}^{(\alpha)}(x) &= \sum_{\substack{\sg: \mu \to \mathbb{Z}_{>0} \\ \sg \text{ non-attacking}}} x^{\sg} \prod_{\substack{u \in \mu \\ \sg(u) = \sg(\down_{\mu}(u))}} \left(1 + \hook^{(\alpha)}_{\mu}(u)\right) \end{align}
\end{thm}

This formula can also be derived from Theorem \ref{thm:hhl-nsym}.

\subsection{Graphs}
We will think of graphs as sets of edges $\{u,v\}$ where $u<v$ on some fixed vertex set $\{1,2,\ldots,n\}$. 

\begin{defn}[\cite{stanley-chromatic}]
The chromatic symmetric function of a graph $G$ is the symmetric function
\begin{align}
\chromatic_G(x) &= \sum_{\substack{\kappa : V(G) \to \mathbb{Z}_{>0} \\ \kappa \text{ proper coloring}}} x^{\kappa} 
\end{align}
where
\begin{align}
x^{\kappa} &= \prod_{v \in V(G)} x_{\kappa(v)} .
\end{align}
Here a map $\kappa : V(G) \to \mathbb{Z}_{>0}$ is said to be a proper coloring if $\kappa(u) \neq \kappa(v)$ whenever $\{u,v\} \in G$.
\end{defn}

Evaluating $\chromatic_G(x)$ at $x_1=x_2=\ldots=x_k=1$ and $x_{k+1}=x_{k+2}=\ldots=0$ recovers the value of the chromatic polynomial of $G$ at the positive integer $k$. In \cite{shareshian-wachs}, Shareshian and Wachs define a $t$-analogue of Stanley's chromatic symmetric function. 

\begin{defn}[\cite{shareshian-wachs}]
Given a proper coloring $\kappa$ of a graph $G$, let 
\begin{align}
\asc(\kappa) = \# \{\{u,v\} \in G : u < v,  \kappa(u) < \kappa(v)\}. 
\end{align}
The chromatic quasisymmetric function of $G$ is defined to be
\begin{align}
\chromatic_G(x;t) &= \sum_{\substack{\kappa : V(G) \to \mathbb{Z}_{>0} \\ \kappa \text{ proper coloring}}} x^{\kappa} t^{\asc(\kappa)}
\end{align} 
\end{defn}

Although $\chromatic_G(x;t)$ is not generally symmetric in the $x_i$'s, it will be symmetric in all the cases we address.

\section{Expansions of integral form Macdonald polynomials}
\label{sec:integral-form}

In this section, we give a formula for the integral form Macdonald polynomials as a weighted sum of chromatic quasisymmetric functions. Then we use formulas for the Schur and power sum expansion of chromatic quasisymmetric functions, proved by \cite{gasharov} and \cite{ath-power}, respectively, to give Schur and power sum formulas for the integral form Macdonald polynomials. First, we need to define certain graphs.

\begin{defn}
Given a partition $\mu \vdash n$, we number the cells of $\mu$ in \emph{reading order} (left to right, top to bottom) with entries $1, 2, \ldots, n$. The \emph{attacking graph} $G_{\mu}$ has vertex set $1,2,\ldots,n$ and edges $\{u, v\}$ if and only if cells $u < v$ are attacking in $\mu$. The \emph{augmented attacking graph} $G_{\mu}^{+}$ consists of the edges of $G_{\mu}$ along with the extra edges $\{u, \down_{\mu}(u)\}$ for every cell $u$ not in the bottom row of $\mu$. 
\end{defn}

\subsection{Chromatic quasisymmetric functions}

\begin{thm}
\label{thm:integral-form-chromatic}
Let $n(\lambda) = \sum_{i=1}^{\ell(\lambda)} (i-1) \lambda_i = \sum_{u \in \lambda} \leg_{\lambda}(u)$. Then
\begin{align}
J_{\mu^{\prime}}(x;q,t) &= t^{-n(\mu^{\prime}) + \binom{\mu_1}{2}} (1-t)^{\mu_1}  \sum_{G_{\mu} \subseteq H \subseteq G_{\mu}^{+}} \chromatic_{H}(x;t) \\
&\times \prod_{\{u,v\} \in H \setminus G_{\mu}} - \left(1 - q^{\leg_{\mu}(u) + 1} t^{\arm_{\mu}(u)}\right) \nonumber \\
&\times \prod_{\{u,v\} \in G_{\mu}^{+} \setminus H} \left(1 - q^{\leg_{\mu}(u)+1}t^{\arm_{\mu}(u)+1}\right) . \nonumber
\end{align}
\end{thm}

Before we prove Theorem \ref{thm:integral-form-chromatic}, we work through the example $\mu = (3,2)$. The cells of $\mu$ are numbered as follows.
\begin{align}
\begin{ytableau}
1 & 2 \\
3 & 4 & 5
\end{ytableau}
\end{align}
Then $G_{\mu}$ is the graph 
\begin{center}
\begin{tikzpicture}
\node (1) {1};
\node (2) [right of=1] {2};
\node (3) [right of=2] {3};
\node (4) [right of=3] {4};
\node (5) [right of=4] {5};
\path
(1) edge node [right] {} (2)
(2) edge node [right] {} (3)
(3) edge node [right] {} (4)
(5) edge[bend right] node [left] {} (3)
(4) edge node [right] {} (5);
\end{tikzpicture}
\end{center}
and $G_{\mu}^{+}$ is the graph obtained by adding edges $\{1,3\}$ and $\{2,4\}$:
\begin{center}
\begin{tikzpicture}
\node (1) {1};
\node (2) [right of=1] {2};
\node (3) [right of=2] {3};
\node (4) [right of=3] {4};
\node (5) [right of=4] {5};
\path
(1) edge node [right] {} (2)
(2) edge node [right] {} (3)
(3) edge node [right] {} (4)
(5) edge[bend right] node [left] {} (3)
(4) edge node [right] {} (5)
(3) edge[bend right] node [left] {} (1)
(4) edge[bend right] node [left] {} (2);
\end{tikzpicture}
\end{center}
We note that $\arm_{\mu}(1) = 1$, $\leg_{\mu}(1) = \arm_{\mu}(2)=\leg_{\mu}(2) = 0$, and $n(\mu^{\prime}) = 2+2=4$. The formula in Theorem \ref{thm:integral-form-chromatic} gives
\begin{align}
J_{\mu}(x;q,t) &= t^{-1} \left[ \left(1 - qt^2 \right) \left( 1 - qt \right) \chromatic_{G_{\mu}} (x;t) \right. \\
&- \left( 1 - qt \right) \left( 1 - qt \right) \chromatic_{G_{\mu} \cup \{1,3\}} (x;t) \nonumber \\
&- \left(1 - qt^2 \right) \left( 1 - q \right) \chromatic_{G_{\mu} \cup \{2,4\}} (x;t)\nonumber \\
&+ \left. \left( 1 - qt \right) \left( 1 - q \right) \chromatic_{G_{\mu}^{+}} (x;t) \right] \nonumber
\end{align}
where $G_{\mu} \cup \{u,v\}$ means we add the edge $\{u, v\}$ to $G_{\mu}$. Note that, after distributing the $t^{-1}$ term, the coefficients in this expansion are, in general, only in $\mathbb{Z}(q,t)$ but not $\mathbb{Z}[q,t]$. It is only after adding all the terms together that we get polynomial coefficients.

\begin{proof}[Proof of Theorem \ref{thm:integral-form-chromatic}]
First, we precisely state Haglund, Haiman, and Loehr's formula for $J_{\mu^{\prime}}(x;q,t)$. Given an assignment $\sg$ of a positive integer to each cell in the diagram of $\mu$, recall that $\sg$ is non-attacking if $\sg(u) \neq \sg(v)$ for each pair of cells $u$ and $v$ that either share a row or $v$ is one row below $u$ and strictly to $u$'s left. (We say such pairs of cells are \emph{attacking}.) We also let
\begin{itemize}
\item $\down_{\mu}(u)$ be the cell just below the cell $u$, if such a cell exists,
\item $\Des(\sg, \mu)$ be the set of cells with $\sg(u) > \sg(\down_{\mu}(u))$, 
\item $\maj(\sg, \mu) = \sum_{u \in \Des(\sg, \mu)} (\leg_{\mu}(u) + 1)$, 
\item $\Inv(\sg, \mu)$ be the set of attacking pairs of cells $u$ and $v$ such that $u$ appears before $v$ in reading order (as the cells are processed top down and left to right) and $\sg(u) > \sg(v)$, and
\item $\inv(\sg, \mu) = \# \Inv(\sg, \mu) - \sum_{u \in \Des(\sg, \mu)} \arm_{\mu}(u)$.
\end{itemize}
Then we have
\begin{align}
\label{hhl-proof}
J_{\mu^{\prime}}(x;q,t) &= \sum_{\substack{\sg: \mu \to \mathbb{Z}_{>0} \\ \sg \text{ non-attacking}}} x^{\sg} q^{\maj(\sg, \mu)} t^{n(\mu^{\prime}) - \inv(\sg, \mu)} \\
&\times \prod_{\substack{u \in \mu \\ \sg(u) = \sg(\down_{\mu}(u))}} \left( 1 - q^{\leg_{\mu}(u)+1} t^{\arm_{\mu}(u)+1} \right) \nonumber \\
&\times \prod_{\substack{u \in \mu \\ \sg(u) \neq \sg(\down_{\mu}(u))}} \left( 1 - t \right) . \nonumber
\end{align}
where each cell in the bottom row of $\mu$ contributes to the last product. 

We need to modify the power of $t$ slightly. Recall that $n(\mu^{\prime}) = \sum_{u \in \mu} \arm_{\mu}(u)$. The number of possible attacking pairs in a filling of shape $\mu$ is
\begin{align}
n(\mu^{\prime}) + \sum_{u \in \mu : u \notin \mu_1} \arm_{\mu}(u) = 2 n(\mu^{\prime}) - \binom{\mu_1}{2} .
\end{align}
Let $\Coinv(\sg, \mu)$ be the set of attacking pairs $u$ and $v$ such that $\sg(u) < \sg(v)$. Since $\sg$ is a non-attacking filling, we have 
\begin{align}
\# \Inv(\sg, \mu) + \# \Coinv(\sg, \mu) &= 2 n(\mu^{\prime}) - \binom{\mu_1}{2}
\end{align}
which implies
\begin{align}
n(\mu^{\prime}) - \inv(\sg, \mu) &= n(\mu^{\prime}) - \# \Inv(\sg, \mu) + \sum_{u \in \Des(\sg, \mu)} \arm_{\mu}(u) \\
&= \# \Coinv(\sg, \mu) - n(\mu^{\prime}) + \binom{\mu_1}{2} + \sum_{u \in \Des(\sg, \mu)} \arm_{\mu}(u) .
\end{align}
Hence, for a fixed non-attacking filling $\sg$, the contribution of $\sg$ to \eqref{hhl-proof} is
\begin{align}
\label{hhl-chi}
& x^{\sg} t^{-n(\mu^{\prime}) + \binom{\mu_1}{2}} (1-t)^{\mu_1} 
\prod_{u, v \text{ attacking}} t^{\chi(\sg(u) < \sg(v))} \\
&\prod_{v = \down_{\mu}(u)} \left( 1 - q^{\leg_{\mu}(u)+1} t^{\arm_{\mu}(u)+1} \right)^{\chi(\sg(u) = \sg(v))} \\ 
&\times (1-t)^{\chi(\sg(u) \neq \sg(v))} \left( q^{\leg_{\mu}(u)+1} t^{\arm_{\mu}(u)} \right)^{\chi(\sg(u)>\sg(v)} \nonumber .
\end{align}
where $\chi$ evaluates to 1 if the inner statement is true and 0 if the statement is false.

For any non-attacking filling $\sg$ of $\mu$, let $\kappa$ be the coloring of the vertices $1,2,\ldots,n$ given by $\kappa(u) = \sg(u)$. We will show that the contribution of $\sg$ to \eqref{hhl-proof} is equal to the contribution of $\kappa$ to the expression in Theorem \ref{thm:integral-form-chromatic}: 
\begin{align}
\label{chromatic-proof}
 &t^{-n(\mu^{\prime}) + \binom{\mu_1}{2}} (1-t)^{\mu_1}  \sum_{G_{\mu} \subseteq H \subseteq G_{\mu}^{+}} \chromatic_{H}(x;t)  \\ \nonumber
 &\times \prod_{\{u,v\} \in H \setminus G_{\mu}} - \left(1 - q^{\leg_{\mu}(u) + 1} t^{\arm_{\mu}(u)}\right) \\
&\times \prod_{\{u,v\} \in G_{\mu}^{+} \setminus H} \left(1 - q^{\leg_{\mu}(u)+1}t^{\arm_{\mu}(u)+1}\right) . \nonumber
\end{align} 
Clearly the term  $x^{\kappa} t^{-n(\mu^{\prime}) + \binom{\mu_1}{2}} (1-t)^{\mu_1}$ appears equally in both expressions. For attacking pairs $u$ and $v$ in $\mu$, we get a contribution of $t^{\chi(\kappa(u) < \kappa(v))}$ to \eqref{chromatic-proof} coming from the definition of $\chromatic_{H}(x;t)$ for any $G_{\mu} \subseteq H \subseteq G_{\mu}^{+}$. 

All that remains is to consider the cells $u,v$ where $v = \down_{\mu}(u)$. We get different contributions to \eqref{chromatic-proof} in each of three situations: when $\sg(u) > \sg(v)$, when $\sg(u) = \sg(v)$, and when $\sg(u) < \sg(v)$. We will show these contributions match the corresponding contributions to \eqref{hhl-chi}\footnote{One can consider the two factors in \eqref{chromatic-proof} as the solutions to a system of two unknowns and three equations. Of course, such a solution is not guaranteed to exist -- the fact that it does exist in this case is fortunate, although this phenomenon occurs often in the theory of Macdonald polynomials.}.

If $\sg(u) < \sg(v)$, we get a contribution of $1-t$ to \eqref{hhl-chi}. We get a contribution of $ \left(1 - q^{\leg_{\mu}(u)+1}t^{\arm_{\mu}(u)+1}\right)$  for each $H$ that does not contain $\{u,v\}$ and $- t \left(1 - q^{\leg_{\mu}(u) + 1} t^{\arm_{\mu}(u)}\right)$ for each $H$ which does contain $\{u,v\}$, since we have to include the factor from \eqref{chromatic-proof} as well as the extra $t$ that comes from a new ascent. Grouping on this condition and summing the contributions, we get
\begin{align}
\left(1 - q^{\leg_{\mu}(u)+1}t^{\arm_{\mu}(u)+1}\right) - t \left(1 - q^{\leg_{\mu}(u) + 1} t^{\arm_{\mu}(u)}\right) = 1-t
\end{align}
as desired. 

If $\sg_u = \sg_v$, we get a term of $1-q^{\leg_{\mu}(u)+1} t^{\arm_{\mu}(u)+1}$ in \eqref{hhl-chi}. From \eqref{chromatic-proof}, we get $\left(1 - q^{\leg_{\mu}(u)+1} t^{\arm_{\mu}(u)+1} \right)$  for each $H$ that does not contain $\{u,v\}$ and no contribution from the other graphs, since $\kappa$ is not proper for those graphs. These factors are clearly equal.

Finally, if $\sg(u) > \sg(v)$ we get $q^{\leg_{\mu}(u)+1}t^{\arm_{\mu}(u)}(1-t)$ from \eqref{hhl-chi}. We get a contribution of $ \left(1 - q^{\leg_{\mu}(u)+1}t^{\arm_{\mu}(u)+1}\right)$  for each $H$ that does not contain $\{u,v\}$ and $- t \left(1 - q^{\leg_{\mu}(u) + 1} t^{\arm_{\mu}(u)}\right)$ for each $H$ which does contain $\{u,v\}$. Grouping on this condition and summing the two contributions, we get
\begin{align}
\left(1 - q^{\leg_{\mu}(u)+1}t^{\arm_{\mu}(u)+1}\right) - \left(1 - q^{\leg_{\mu}(u) + 1} t^{\arm_{\mu}(u)}\right) = q^{\leg_{\mu}(u)+1}t^{\arm_{\mu}(u)}(1-t)
\end{align}
as desired.

\end{proof}

\subsection{Schur functions}

In \cite{shareshian-wachs}, Shareshian and Wachs provide a formula for the Schur expansion of $\chromatic_G(x;t)$ for certain graphs $G$. By applying Theorem \ref{thm:integral-form-chromatic} to their Schur formula, we obtain a Schur formula for the integral form Macdonald polynomials. Since the Schur expansion of $J_{\mu}(x;q,t)$ is not positive, our formulas cannot be positive. In general, we will have some cancelation. In order to state our formula, we need to consider certain tableaux. 

\begin{defn}
Given two partitions $\mu, \lambda \vdash n$, we again number the cells of $\mu$ with entries $1,2,\ldots,n$ in reading order. An \emph{integral form tableau} $T$ of type $\mu$ and shape $\lambda$, is a bijection $T : \lambda \to \{1,2,\ldots,n\}$ such that 
\begin{itemize}
\item the entries in each row increase from left to right,
\item if $v$ is immediately to the right of $u$ then $\{u, v\} \notin G_{\mu}$, and
\item if $v$ is immediately below $u$ and $u < v$ then $\{u, v\} \in G_{\mu}^{+}$.
\end{itemize}
We let $\itab_{\mu}$ denote the collection of all integral form tableaux of type $\mu$ (of varying shapes $\lambda$). 
\end{defn}

For example,
\begin{align}
\label{integral-tableau}
\begin{ytableau}
2 \\
3 &  5 \\
1 & 4 & 6
\end{ytableau}
\end{align}
is an integral form tableau of type $(2,2,2)$ and shape $(3,2,1)$. For reference, the reading order labeling of $(2,2,2)$ is
\begin{align}
\begin{ytableau}
1 & 2 \\
3 & 4 \\
5 & 6
\end{ytableau}
\end{align} 
and one can check that this tableau satisfies all the conditions described above. Note that a given tableau can be an integral form tableau for many different types $\mu$. Given such a tableau $T$, we construct a $q,t$-weight for the tableau, denoted $\weight_{\mu}(T)$, as follows. 

\begin{defn}
For each $u < v$ such that $\{u, v\} \in G_{\mu}^{+} \setminus G_{\mu}$, i.e.\ $v = \down_{\mu}(u)$,  we define a factor $\weight_{\mu}(T,u)$ by
\begin{itemize}
\item $\weight_{\mu}(T,u) =  t^{-\arm_{\mu}(u)} \left(1 - q^{\leg_{\mu}(u)+1}t^{\arm_{\mu}(u)+1} \right)$ if $u$ appears immediately to the left of $v$ in $T$,
\item $\weight_{\mu}(T,u) = -t^{-\arm_{\mu}(u)+1}\left(1 - q^{\leg_{\mu}(u)+1} t^{\arm_{\mu}(u)} \right)$ if $u$ appears immediately above $v$ in $T$,
\item $\weight_{\mu}(T,u) = t^{-\arm_{\mu}(u)}(1-t)$ if $u$ appears in some row above $v$ in $T$ but not immediately above $v$ in $T$, and
\item $\weight_{\mu}(T,u) = q^{\leg_{\mu}(u)+1}(1-t)$ otherwise.
\end{itemize}
We also let $\Inv_{\mu}(T)$ equal the set of edges $\{u,v\} \in G_{\mu}$ such that $u$ appears in a row strictly above $v$'s row and $\inv_{\mu}(T) = \# \Inv_{\mu}(T)$. Then we define
\begin{align}
\wt_{\mu}(T) &= t^{\inv_{\mu}(T)} \prod_{\{u,v\} \in G_{\mu}^{+} \setminus G_{\mu}} \weight_{\mu}(T,u) .
\end{align}
\end{defn} 

Let us consider the tableau $T$ of type $(2,2,2)$ and shape $(3,2,1)$ depicted in \eqref{integral-tableau}. The edges that are in $G_{\mu}^{+}$ but not in $G_{\mu}$ are $\{1,3\}$, $\{2,4\}$, $\{3,5\}$, and $\{4,6\}$. These contribute the following factors, respectively:
\begin{align}
\weight_{\mu}(T,1) &= q^{\leg_{\mu}(1)+1} \left(1-t \right) = q\left(1-t\right) \\
\weight_{\mu}(T,2) &= t^{-\arm_{\mu}(2)}(1 - t) = 1-t  \\
\weight_{\mu}(T,3)  &= t^{-\arm_{\mu}(3)} \left(1 - q^{\leg_{\mu}(3)+1}t^{\arm_{\mu}(3)+1} \right) = t^{-1} \left(1 - q^2t^2 \right) \\
\weight_{\mu}(T,4) &= t^{-\arm_{\mu}(4)} \left(1 - q^{\leg_{\mu}(4)+1} t^{\arm_{\mu}(4)+1} \right) = 1 - q^2 t .
\end{align}
Multiplying the four terms together and then by $t^{\inv_{\mu}(T)} = t^3$, we get 
\begin{align}
\wt_{\mu}(T) &= qt^2 \left(1-t\right)^2 \left(1-q^2t\right) \left(1-q^2t^2\right) .
\end{align}

\begin{cor}
\label{cor:integral-form-schur}
For any partition $\mu \vdash n$, 
\begin{align}
J_{\mu^{\prime}}(x;q,t) &= (1-t)^{\mu_1}   \sum_{T \in \itab_{\mu}} \weight_{\mu}(T) s_{\shape(T)}
\end{align}
where the sum is over all integral form tableaux of type $\mu$.
\end{cor}

As an example, we compute the coefficient of $s_{2,2}$ in the Schur expansion of $J_{3,1}(x;q,t)$, so $\mu = (2,1,1)$. The only attacking pair in the reading order labeling of $(2,1,1)$ is $\{3,4\}$. We also note that 1 is immediately above 2 and 2 is immediately above 3 in this labeling. The integral form tableaux of type $(2,1,1)$ and shape $(2,2)$ are
\begin{align}
\begin{ytableau}
2 & 4 \\
1 & 3
\end{ytableau}
\hspace{15pt}
\begin{ytableau}
2 & 3 \\
1 & 4
\end{ytableau}
\hspace{15pt}
\begin{ytableau}
1 & 4 \\
2 & 3
\end{ytableau}
\hspace{15pt}
\begin{ytableau}
1 & 3 \\
2 & 4
\end{ytableau}
\end{align}
and their respective weights are $q(1-t)^2 $, $qt(1-t)(1-q^2 t)$, $ -t(1-q)(1-q^2t)$, and $ -q^2 t^2 (1-q) (1-t)$. Summing these weights and multiplying by $(1-t)^2$, we see that the coefficient
\begin{align}
\left. J_{3,1}(x;q,t) \right|_{s_{2,2}} &= (1-t)^2 (q-t) (1-qt) (1+qt) .
\end{align}

\begin{proof}[Proof of Corollary \ref{cor:integral-form-schur}]
First, we need to recall the Schur expansion of a chromatic quasisymmetric function given in \cite{shareshian-wachs}. If $G$ is the incomparability graph of a $(3+1)$-free poset, we define a \emph{$G$-tableau} of shape $\lambda$ to be a bijective filling of $\lambda$ with integers $1,2,\ldots, |\lambda|$ such that
\begin{itemize}
\item if $u$ is immediately left of $v$, then $\{u,v\} \notin G$, and
\item if $u$ is immediately above $v$, then either $u > v$ or $\{u,v\} \in G$. 
\end{itemize}
Given such a tableau $T$, we define $\inv_{G}(T)$ to be the number of $\{u,v\} \in G$ such that $u$ appears in a row strictly above the row containing $v$. In \cite{shareshian-wachs}, Shareshian and Wachs extend a result of Gasharov \cite{gasharov} to prove that
\begin{align}
\chromatic_G(x;t) &= \sum_{\text{$G$-tableaux } T} t^{\inv_G(T)} s_{\shape(T)} 
\end{align}
for all such graphs $G$. Applying this result to Theorem \ref{thm:integral-form-chromatic}, we have
\begin{align}
\label{integral-schur-pf}
J_{\mu^{\prime}}(x;q,t) &= t^{-n(\mu^{\prime}) + \binom{\mu_1}{2}} (1-t)^{\mu_1}  \sum_{G_{\mu} \subseteq H \subseteq G_{\mu}^{+}}  \prod_{\{u,v\} \in H \setminus G_{\mu}} - \left(1 - q^{\leg_{\mu}(u) + 1} t^{\arm_{\mu}(u)}\right)  \\
&\times \prod_{\{u,v\} \in G_{\mu}^{+} \setminus H} \left(1 - q^{\leg_{\mu}(u)+1}t^{\arm_{\mu}(u)+1}\right) \nonumber
\sum_{\text{$H$-tableaux } T} t^{\inv_H(T)} s_{\shape(T)} .
\end{align}
It is not hard to check that integral form tableaux $T$ of type $\mu$ are exactly the tableaux that are $H$-tableaux for (at least) one graph $H$ such that $G_{\mu} \subseteq H \subseteq G_{\mu}^{+}$. The main idea of the proof is to switch the order of the sums in \eqref{integral-schur-pf}.  

We fix an integral form tableau $T$ of type $\mu$. We wish to calculate the contribution of $T$ to \eqref{integral-schur-pf}. The difficulty is that $T$ may be an $H$-tableau for more than one graph $H$. First, we know that each graph $H$ contains $G_{\mu}$, so we need to include $t^{\inv_{G_{\mu}}(T)} = t^{\inv_{\mu}(T)}$. The other contributions will come from edges $\{u,v\} \in G_{\mu}^{+} \setminus G_{\mu}$. We consider the possible configurations case by case and show that they match $\wt_{\mu}(T,u)$.

First, consider the case when $u$ appears immediately left of $v$ in $T$. Then $H$ must not contain the edge $\{u,v\}$, so we must choose the factor $1 - q^{\leg_{\mu}(u)+1}t^{\arm_{\mu}(u)+1}$ from the second product in \eqref{integral-schur-pf}. The $t^{-\arm_{\mu}(u)}$ factor comes from the power of $t$ at the front of \eqref{integral-schur-pf}. Multiplying these together, we obtain $\wt_{\mu}(T,u)$. 

Next, if $u$ appears immediately above $v$ in $T$, then we must have $\{u,v\} \in H$, which forces us to take the factor $-\left(1 - q^{\leg_{\mu}(u) + 1} t^{\arm_{\mu}(u)}\right)$. Again, we take $t^{-\arm_{\mu}(u)}$ from the front of \eqref{integral-schur-pf}. We also get another power of $t$ in $t^{\inv_{H}(T)}$ that does not come from $t^{\inv_{\mu}(T)}$, so we multiply by $t$ to get $\wt_{\mu}(T,u)$. 

If $u$ appears in a row above $v$ but not immediately above $v$, then $\{u,v\} \in H$ and $\{u,v\} \notin H$ are both possible. Since $T$ appears in both of these cases, we sum the weights from each of the two cases in \eqref{integral-schur-pf}. If $\{u,v\} \in H$, the contribution to \eqref{integral-schur-pf}
\begin{align}
-t^{-\arm_{\mu}+1} \left(1 - q^{\leg_{\mu}(u) + 1} t^{\arm_{\mu}(u)}\right) .
\end{align}
If $\{u,v\} \notin H$, then the contribution is 
\begin{align}
t^{-\arm_{\mu}} \left(1 - q^{\leg_{\mu}(u) + 1} t^{\arm_{\mu}(u)+1}\right) .
\end{align}
Adding these two terms, we obtain $\weight_{\mu}(T,u) = t^{-\arm_{\mu}(u)}(1-t)$.

Finally, if none of these cases apply, then $u$ is in a row strictly south of $v$. Since we never have an additional inversion in this case,  the total contribution to \eqref{integral-schur-pf} is
\begin{align}
&-t^{-\arm_{\mu}} \left(1 - q^{\leg_{\mu}(u) + 1} t^{\arm_{\mu}(u)}\right) + t^{-\arm_{\mu}} \left(1 - q^{\leg_{\mu}(u) + 1} t^{\arm_{\mu}(u)+1}\right) \\
&= q^{\leg_{\mu}(u)+1}(1-t)
\end{align}
which is equal to the desired weight. This allows us to rewrite \eqref{integral-schur-pf} as the statement in the corollary.
\end{proof}

We noted previously that the coefficients in Theorem \ref{thm:integral-form-chromatic} are only rational functions, instead of polynomials, in $q$ and $t$. Here we prove that this problem disappears for Corollary \ref{cor:integral-form-schur}.

\begin{prop}
For any integral form tableau $T$ of type $\mu$, 
\begin{align}
\weight_{\mu}(T) \in \mathbb{Z}[q,t].
\end{align} 
\end{prop}

\begin{proof}
We consider $\{u,v\} \in G_{\mu}^{+} \setminus G_{\mu}$ such that 
\begin{enumerate}
\item $u$ appears immediately left of $v$ in $T$, 
\item $u$ appears immediately above $v$ in $T$, or
\item $u$ appears above $v$ in $T$ but not immediately above $v$ in $T$.
\end{enumerate}
By definition, these are the only edges that have a negative power of $t$ in $\weight_{\mu}(T)$. Their  relevant contributions are $t^{-\arm_{\mu}(u)}$, $t^{-\arm_{\mu}(u)+1}$, and $t^{-\arm_{\mu}(u)}$, respectively. Let 
\begin{align}
\armset_{\mu}(u) &= \{w > u : w \text{ is in the same row of $\mu$ as $u$} \}.
\end{align}
It is clear from the definition that the cardinality of $\armset_{\mu}(u)$ is $\arm_{\mu}(u)$. Furthermore, we define
\begin{align}
\inv_{\mu}(T,u) &= \# \{w \in \armset_{\mu}(u) : \{u,w\} \in \Inv_{\mu}(T)\} \\
&+  \# \{w \in \armset_{\mu}(u) : \{w,v\} \in \Inv_{\mu}(T)\} \nonumber.
\end{align}
Our goal is to show that, in each of the three cases above, $\inv_{\mu}(T,u) \geq \arm_{\mu}(u)$. Since each inversion in $T$ contributes to exactly one $\inv_{\mu}(T,u)$, this will prove the claim. The key to our argument is that $w \in \armset_{\mu}(u)$ cannot share a row in $T$ with either $u$ or $v$.

Consider case 1, i.e.\ $u$ appears immediately left of $v$ in $T$. By the definition of an integral form tableau, $w \in \armset_{\mu}(u)$ may not appear in the row containing $u$ and $v$ in $T$. If it appears above this row, then $\{w,v\}$ is an inversion; if it appears below this row, then $\{u,w\}$ is an inversion. Therefore $\inv_{\mu}(T,u) = \arm_{\mu}(u)$. 

In case 2, $w$ may not appear in the row containing $u$ or the row containing $v$ in $T$, and the argument from case 1 applies.

In case 3, we have three possible configurations. Listing the cells in their order of appearance from top to bottom we either have $wuv$, $uwv$, or $uvw$. By the same logic as above, each of these three orders contributes (at least) one inversion, so $\inv_{\mu}(T,u) \geq \arm_{\mu}(u)$. 
\end{proof}

\subsection{Power sum symmetric functions}

In a similar manner, we can use Theorem \ref{thm:integral-form-chromatic} along with the power sum expansion of $\chromatic_G$, conjectured in \cite{shareshian-wachs} and proved in \cite{ath-power}, to give a power sum formula for the integral form Macdonald polynomial.

First, we recall the formula for the power sum expansion of $\chromatic_H(x;t)$ for our graphs $H$. We translate the formula from the poset setting to the graph setting for the sake of consistency. Given a partition $\lambda \vdash n$ of length $k$ and a permutation $\sg \in \S_n$ in one-line notation, we break $\sg$ into $k$ blocks of lengths $\lambda_1, \lambda_2, \ldots, \lambda_k$ from left to right. For a fixed $\lambda$, we say that $\sg$ has an \emph{$H$-descent} at position $i$ if 
\begin{itemize}
\item $i$ and $i+1$ are in the same block,
\item $\sg_i > \sg_{i+1}$, and
\item  the edge $\{\sg_{i+1}, \sg_i\} \notin H$. 
\end{itemize}
Similarly, we say that $\sg$ has a \emph{nontrivial left-to-right $H$-maximum} at position $j$ if
\begin{itemize}
\item $j$ is not the first position in its block, 
\item $\sg_i < \sg_j$ for each $i < j$ in $j$'s block, and 
\item $\{\sg_i, \sg_j\} \notin H$ for each $i < j$ in $j$'s block.
\end{itemize}
Let $\mathcal{N}_{\lambda}(H)$ be the set of permutations $\sg \in \S_n$ without $H$-descents and with no nontrivial left-to-right $H$-maxima. 

\begin{thm}[\cite{ath-power}]
\label{thm:ath-power}
For any graph $H$ which is the incomparability graph of a natural unit order, we have
\begin{align}
\omega \chromatic_H(x;t) &= \sum_{\lambda \vdash n} \frac{p_{\lambda}}{z_{\lambda}} \sum_{\sg \in \mathcal{N}_{\lambda}(H)} t^{\inv_{H}(\sg)}
\end{align}
where $\inv_H(\sg)$ is the number of pairs of indices $i < j$ such that $\sg_i > \sg_j$ and $\{\sg_i, \sg_j\} \in H$.
\end{thm}

Now we must define a weight for each $\sg \in \mathcal{N}_{\lambda}(G_{\mu}^{+})$.

\begin{defn}
\label{defn:p-weight}
For partitions $\mu, \lambda \vdash n$ and $\sg \in \mathcal{N}_{\lambda}(G_{\mu}^{+})$, we define $\wt_{\mu, \lambda}(\sg)$ to be the product over all $\{u,v\} \in G_{\mu}^{+} \setminus G_{\mu}$ where the term corresponding to $\{u,v\}$ is
\begin{enumerate}
\item $-t (1-q^{\leg_{\mu}(u)+1} t^{\arm_{\mu}(u)})$ if $u$ and $v$ form a $G_{\mu}$-descent in $\sg$, \\
\item $-(1-q^{\leg_{\mu}(u)+1} t^{\arm_{\mu}(u)})$ if $v$ is a nontrivial left-to-right $G_{\mu}$-maximum in $\sg$, \\
\item $1-t$ if otherwise and $u$ and $v$ form a $G_{\mu}$-inversion in $\sg$, and
\item $q^{\leg_{\mu}(u)+1} t^{\arm_{\mu}(u)} (1-t)$ if otherwise.
\end{enumerate}
\end{defn}

\begin{cor}
\begin{align}
\omega J_{\mu^{\prime}}(x; q, t) &= t^{-n(\mu^{\prime}) + \binom{\mu_1}{2}} (1-t)^{\mu_1} \sum_{\lambda \vdash n} \frac{p_{\lambda}}{z_{\lambda}} \sum_{\sg \in \mathcal{N}_{\lambda}(G_{\mu}^{+})} \wt_{\mu, \lambda}(\sg) 
\end{align}
\end{cor}


\begin{proof}
By Theorem \ref{thm:integral-form-chromatic} and \ref{thm:ath-power}, 
\begin{align}
\omega J_{\mu^{\prime}}(x; q, t) &=  t^{-n(\mu^{\prime}) + \binom{\mu_1}{2}} (1-t)^{\mu_1}  \\ &\sum_{G_{\mu} \subseteq H \subseteq G_{\mu}^{+}}
\prod_{\{u,v\} \in H \setminus G_{\mu}} - \left(1 - q^{\leg_{\mu}(u) + 1} t^{\arm_{\mu}(u)}\right) \nonumber \\
&\times \prod_{\{u,v\} \in G_{\mu}^{+} \setminus H} \left(1 - q^{\leg_{\mu}(u)+1}t^{\arm_{\mu}(u)+1}\right) . \nonumber \\
&\times \sum_{\lambda \vdash n} \frac{p_{\lambda}}{z_{\lambda}} \sum_{\sg \in \mathcal{N}_{\lambda}(H)} t^{\inv_{H}(\sg)} .
\end{align}
Now we consider each edge $\{u, v\} \in G_{\mu}^{+} \setminus G_{\mu}$ along with Definition \ref{defn:p-weight}. If $\sg \in \mathcal{N}_{\lambda}(H)$ for any of the graphs $H$ it must be in $\mathcal{N}_{\lambda}(G_{\mu}^{+})$. If we see a $G$-descent $vu$, then we know that $\{u,v\} \in H$, so we multiply the factor $ - \left(1 - q^{\leg_{\mu}(u) + 1} t^{\arm_{\mu}(u)}\right)$ as well as by a $t$, since we get an $H$-inversion between $u$ and $v$. This yields the factor in (1) of Definition \ref{defn:p-weight}. The second possibility is that $v$ is a nontrivial left-to-right $G_{\mu}$-maximum. Again, this forces $\{u,v\} \in H$ and we get the factor appearing in (2) in Definition \ref{defn:p-weight}. If neither of these situations occurs, then we have freedom to either include or exclude $\{u,v\}$ from $H$. If $u$ is to the left of $v$ in $\sg$, this results in simply adding the factors
\begin{align}
 - \left(1 - q^{\leg_{\mu}(u) + 1} t^{\arm_{\mu}(u)}\right) + \left(1 - q^{\leg_{\mu}(u)+1}t^{\arm_{\mu}(u)+1}\right)
 \end{align} 
 which yields the factor in (4) of Definition \ref{defn:p-weight}. If $u$ is to the right of $v$ in $\sg$, then we create a new $H$-inversion if we include the edge $\{u,v\}$, so we get
 \begin{align}
  - t\left(1 - q^{\leg_{\mu}(u) + 1} t^{\arm_{\mu}(u)}\right) + \left(1 - q^{\leg_{\mu}(u)+1}t^{\arm_{\mu}(u)+1}\right) = 1-t
 \end{align}
 which is the factor in (3).
\end{proof}

\section{Expansions of Jack polynomials}
\label{sec:jack}

In this section, we obtain results for Jack polynomials that are analogous to our results for integral form Macdonald polynomials. Often, the results here are simpler than the results in Section \ref{sec:integral-form}, since they are specializations of our previous results. All results follow directly from the analogous result in Section \ref{sec:integral-form} and the definition of $J_{\mu}^{(\alpha)}(x)$, so we omit proofs.

\subsection{Chromatic symmetric functions}

Recall that 
\begin{align}
\hook^{(\alpha)}_{\mu}(u) = \alpha(\leg_{\mu}(u)+1) + \arm_{\mu}(u).
\end{align}

\begin{thm}
\label{thm:jack-chromatic}
\begin{align}
J_{\mu^{\prime}}^{(\alpha)}(x) &= \sum_{G_{\mu} \subseteq H \subseteq G_{\mu}^{+}} \chromatic_H(x) \prod_{\{u,v\} \in H \setminus G_{\mu}} \left(-\hook^{(\alpha)}_{\mu}(u) \right) \prod_{\{u,v\} \in G_{\mu}^{+} \setminus H} \left(1 + \hook^{(\alpha)}_{\mu}(u)\right)
\end{align}
\end{thm}

Let us consider the example $\mu^{\prime} = (2,2,1)$, so $\mu = (3,2)$. The cells of $\mu$ are numbered as follows.
\begin{align}
\begin{ytableau}
1 & 2 \\
3 & 4 & 5
\end{ytableau}
\end{align}
Then $G_{\mu}$ is the following graph.
\begin{center}
\begin{tikzpicture}
\node (1) {1};
\node (2) [right of=1] {2};
\node (3) [right of=2] {3};
\node (4) [right of=3] {4};
\node (5) [right of=4] {5};
\path
(1) edge node [right] {} (2)
(2) edge node [right] {} (3)
(3) edge node [right] {} (4)
(5) edge[bend right] node [left] {} (3)
(4) edge node [right] {} (5);
\end{tikzpicture} 
\end{center}
The formula in Theorem \ref{thm:jack-chromatic} gives
\begin{align}
J_{\mu^{\prime}}^{(\alpha)}(x) &= \left(-\hook^{(\alpha)}_{\mu}(1) \right) \left(-\hook^{(\alpha)}_{\mu}(2) \right) \chromatic_{G_{\mu}} (x) \\
&+ \left(1+ \hook^{(\alpha)}_{\mu}(1)\right) \left(-\hook^{(\alpha)}_{\mu}(2)\right) \chromatic_{G_{\mu} \cup \{1,3\}} (x) \nonumber \\
&+ \left(-\hook^{(\alpha)}_{\mu}(1) \right) \left(1+ \hook^{(\alpha)}_{\mu}(2) \right) \chromatic_{G_{\mu} \cup \{2,4\}} (x)\nonumber \\
&+ \left(1+ \hook^{(\alpha)}_{\mu}(1)\right) \left( 1+\hook^{(\alpha)}_{\mu}(2) \right) \chromatic_{G_{\mu}^{+}} (x)\nonumber 
\end{align}
We have $\hook^{(\alpha)}_{\mu}(1) = \alpha + 1$ and $\hook^{(\alpha)}_{\mu}(2) = \alpha$, which completes the computation.

\subsection{Schur functions}

In \cite{gasharov}, Gasharov gives a formula for the Schur expansion of $\chromatic_G(x)$ whenever $G$ is the incomparability graph of a (3+1)-free poset. One can check that all graphs appearing in Theorem \ref{thm:jack-chromatic} meet this condition, so we get a Schur function formula for Jack polynomials. First, we define a new weight function on integral form tableaux.

To compute $\weight_{\mu}^{(\alpha)}(T)$, we begin with 1 and multiply by 
\begin{itemize}
\item $1+\hook^{(\alpha)}_{\mu}(u)$ if $u$ appears immediately left of $v$ in $T$, and
\item $-\hook^{(\alpha)}_{\mu}(u)$ if $u$ appears immediately above $v$ in $T$
\end{itemize}
for each $\{u,v\} \in G_{\mu}^{+} \setminus G_{\mu}$.
In the example depicted in \eqref{integral-tableau}, the pairs $\{u,v\}$ are $\{1,3\}$, $\{2,4\}$, $\{3,5\}$, and $\{4,6\}$. Then the weight is
\begin{align}
\weight^{(\alpha)}_{\mu}(T) &= \left(1+\hook^{(\alpha)}_{\mu}(2) \right) \left(1+ \hook^{(\alpha)}_{\mu}(4) \right) = (\alpha+1)(2\alpha+1).
\end{align}

\begin{cor}
For any partition $\mu \vdash n$,
\begin{align}
J_{\mu^{\prime}}^{(\alpha)}(x) &= \sum_{T \in \itab_{\mu}} \weight^{(\alpha)}_{\mu}(T) s_{\shape(T)}
\end{align}
where the sum is over all integral form tableaux of type $\mu$.
\end{cor}

As an example, let us revisit the case where $\lambda = (2,2)$ and $\mu = (2,1,1)$. Again, the integral form tableaux of type $(2,1,1)$ and shape $(2,2)$ are as follows.
\begin{align}
\begin{ytableau}
2 & 4 \\
1 & 3
\end{ytableau}
\hspace{15pt}
\begin{ytableau}
2 & 3 \\
1 & 4
\end{ytableau}
\hspace{15pt}
\begin{ytableau}
1 & 4 \\
2 & 3
\end{ytableau}
\hspace{15pt}
\begin{ytableau}
1 & 3 \\
2 & 4
\end{ytableau}
\end{align}
Their respective weights are $1$, $\left(1+\hook^{(\alpha)}_{\mu}(2)\right)$, $\-\hook^{(\alpha)}_{\mu}(1)\left(1+\hook^{(\alpha)}_{\mu}(2)\right)$, and $-\hook^{(\alpha)}_{\mu}(1)$. Adding these weights together and using the fact that $\hook^{(\alpha)}_{\mu}(1) = \alpha$ and $\hook^{(\alpha)}_{\mu}(2) = 2\alpha$, we obtain
\begin{align}
\left. J_{(3,1)}^{(\alpha)}(x) \right|_{s_{2,2}} &= 1 + 2\alpha + 1 - (2\alpha^2 + \alpha) - \alpha = -2\alpha^2 +  2.
\end{align}

\subsection{Power sum symmetric functions}

We can also use Stanley's result in \cite{stanley-chromatic} on the power sum expansion of $\chromatic_G$ to obtain a power sum formula for Jack polynomials.

\begin{cor}
\label{cor:jack-power}
\begin{align}
J_{\mu^{\prime}}^{(\alpha)}(x)  &= \sum_{H \subseteq G_{\mu}^{+}} (-1)^{|H|} p_{\lambda(H)} \prod_{\{u,v\} \in H \setminus  G_{\mu}} -\hook^{(\alpha)}_{\mu}(u) \\
&= \sum_{H \subseteq G_{\mu}^{+}} (-1)^{\left| H \cap G_{\mu} \right|} p_{\lambda(H)} \prod_{\{u,v\} \in H \setminus  G_{\mu}} \hook^{(\alpha)}_{\mu}(u). \nonumber 
\end{align}
where $|H|$ is the number of edges in $H$ and $\lambda(H)$ is the partition whose parts equal the sizes of the connected components of the graph induced by $H$. 
\end{cor}


For example, if $\mu^{\prime} = (3,1)$ then $\mu = (2,1,1)$, so $G_{\mu} = \{\{3,4\}\}$ and $G_{\mu}^{+} = \{\{1,2\},\{2,3\},\{3,4\}\}$. If we wish to compute the coefficient of $p_{2,2}$, we note that the only $H \subseteq G_{\mu}^{+}$ that has 2 connected components of size 2 is $H = \{\{1,2\}, \{3,4\}\}$. Therefore
\begin{align}
\left. J_{(3,1)}^{(\alpha)}(x) \right|_{p_{2,2}} &= -\hook^{(\alpha)}_{\mu}(1) = -\alpha.
\end{align}

\section{Open problems}

\subsection{Cancelation}
As the examples we have computed show, there is generally cancelation in our Schur and power sum formulas. We would like to reduce some of this cancelation via involutions or other methods. It would also be interesting to quantify how much cancelation occurs in our current formulas. For example, one could hope to rigorously compare the amount of cancelation in our Schur formula with the amount of cancelation that comes from computing the Schur expansion using the inverse Kostka numbers. 

\subsection{Hanlon's Conjecture}
A specific cancelation problem comes from a conjecture of Hanlon. Given any bijection $T_0: \lambda \to \{1,2,\ldots,n\}$, let $\RS(T_0)$ and $\CS(T_0)$ be the row and column stabilizers of $T_0$; these are the subgroups of $\S_n$ that consist of permutations that only permute elements that are in the same row or column of $T_0$, respectively. Hanlon conjectured that there is a function $f: \RS(T_0) \times \CS(T_0) \to \mathbb{N}$ such that
\begin{align}
J_{\lambda}^{(\alpha)}(x) = \sum_{\substack{\sg \in \RS(T_0) \\ \tau \in \CS(T_0)}} \alpha^{f(\sg, \tau)} \epsilon(\tau) p_{\type(\sg \tau)}
\end{align}
where $\epsilon(\tau)$ is the sign of the permutation $\tau$ and the type of a permutation is the partition associated to its conjugacy class. This conjecture has been proved in the $\alpha=2$ case \cite{feray-jack}. It bears some intriguing similarities to Corollary \ref{cor:jack-power}, although we have not been able to clarify these similarities as of yet. We hope that Corollary \ref{cor:jack-power} may lead to a new approach to Hanlon's Conjecture.

\subsection{Specializations}
One classical property of the Jack polynomials is that it recovers other well-known functions by specialization, e.g.\ $\alpha = 1$ leads to a scalar multiple of the Schur functions and $\alpha = 2$ recovers the zonal spherical functions \cite{macdonald}. We would like for these specializations to be evident in our formulas for $J_{\mu}^{(\alpha)}$, but that is not currently the case.

\subsection{Schur positivity}
It is a conjecture of Haglund that
\begin{align}
\frac{J_{\mu}(x;q,q^k)}{(1-q)^n}
\end{align}
is Schur positive for any $\mu \vdash n$ and any positive integer $k$. Similarly, we have noticed that the specialization
\begin{align}
\frac{t^{k (n(\mu^{\prime}))}J_{\mu}(x;t^{-k},t)}{(1-t)^n} 
\end{align}
is Schur positive and palindromic for $k \in \mathbb{N}$. We hope that our work may inspire progress on these conjectures.

\subsection{LLT polynomials}
LLT polynomials were defined by Lascoux, Leclerc, and Thibon in \cite{llt} and are prominent objects in the study of Macdonald polynomials. The LLT polynomials associated to collections of single cells are a close relative of chromatic quasisymmetric functions; in fact, every such LLT polynomial can be written
\begin{align}
\llt_G(x; t) &= \sum_{\kappa : V(G) \to \mathbb{Z}_{>0}} x^{\kappa} t^{\asc(\kappa)}
\end{align}
for some  $G$ which is the incomparability graph of a unit interval order. Note that the only difference between this definition and the definition for $\chromatic_G(x; t)$ is that we do not insist that $\kappa$ is a proper coloring in this case. 

The relationship between $\llt_G(x; t)$ and $\chromatic_G(x; t)$ can also be captured using plethysm \cite{plethysm}. Specifically, if $G$ has $n$ vertices then
\begin{align}
\label{chromatic-llt}
\chromatic_G(x; t) = (t-1)^{-n} \llt_G[(t-1)x; t]
\end{align}
where the brackets imply plethystic substitution and $x$ stands for the plethystic sum $x_1+x_2+\ldots$. This identity is equivalent to Proposition 3.4 in \cite{carlsson-mellit}, which the authors of that paper prove using results in \cite{hhl}. Equation \eqref{chromatic-llt} can be rewritten as
\begin{align}
\llt_G(x; t) &= (t-1)^n \chromatic_G [ x/(t-1); t ] .
\end{align}
As a result, every formula for a chromatic quasisymmetric function (of the incomparability graph of a unit interval order) leads to a plethystic formula for the corresponding LLT polynomial. This is especially nice in the power sum case, as plethysm is easiest to compute in the power sum basis. From Theorem \ref{thm:ath-power}, we get
\begin{align}
\omega \llt_G(x; t) &= (t-1)^n \sum_{\lambda \vdash n} \frac{p_{\lambda}[x/(t-1); t]}{z_{\lambda}} \sum_{\sg \in \mathcal{N}_{\lambda}(G)} t^{\inv_G(\sg)} \\
&= (t-1)^n \sum_{\lambda \vdash n} \frac{p_{\lambda}}{(t^{\lambda_1}-1) (t^{\lambda_2}-1) \ldots z_{\lambda}} \sum_{\sg \in \mathcal{N}_{\lambda}(G)} t^{\inv_G(\sg)} \\
\label{llt}
&= \sum_{\lambda \vdash n} \frac{(t-1)^{n-\ell(\lambda)} p_{\lambda}}{[\lambda_1]_t [\lambda_2]_t \ldots z_{\lambda}} \sum_{\sg \in \mathcal{N}_{\lambda}(G)} t^{\inv_G(\sg)}.
\end{align}
where we use the usual notation for the $t$-analogue, $[n]_t = 1 + t + \ldots + t^{n-1}$. Furthermore, Shareshian and Wachs proved that the sum $\sum_{\sg \in \mathcal{N}_{\lambda}(G)} t^{\inv_G(\sg)}$ is divisible by the product $[\lambda_1]_t [\lambda_2]_t \ldots$ \cite{shareshian-wachs}, which we describe below.

Given a partition $\lambda \vdash n$ and a graph $G$ on $n$ vertices that is the incomparability graph of a unit interval order, let $\widetilde{\mathcal{N}}_{\lambda}(G)$ be the set of permutations $\sg \in \S_n$ such that when we break $\sg$ (presented in one-line notation) into segments of lengths $\lambda_1, \lambda_2, \ldots$ then
\begin{outline}
\1 the leftmost entry in each segment is the smallest entry in the segment, and
\1 within each segment, if we see consecutive entries $\sg_i$, $\sg_{i+1}$ such that $\sg_i < \sg_{i+1}$ then we must have $\{\sg_i, \sg_{i+1}\} \notin G$. 
\end{outline}
We can use Proposition 7.8 in \cite{shareshian-wachs} to rewrite \eqref{llt} as 
\begin{align}
\omega \llt_G(x; t) &=  \sum_{\lambda \vdash n} \frac{(t-1)^{n-\ell(\lambda)} p_{\lambda}}{z_{\lambda}} \sum_{\sg \in \widetilde{\mathcal{N}}_{\lambda}(G)} t^{\inv_G(\sg)}.
\end{align}

It would be valuable to see if there is a deeper relationship between LLT polynomials (possibly those that do not correspond to collections of single cells) and chromatic quasisymmetric functions.
\subsection{A non-symmetric analogue}


The \emph{integral form non-symmetric Macdonald polynomials} $\{ \mathcal{E}_{\gamma}(x;q,t) : \gamma \in \mathbb{N}^n\}$ are polynomials in $x_1, x_2, \ldots, x_n$ with coefficients in $\mathbb{Q}(q,t)$ that form a basis for the polynomial ring in variables $x_1, x_2, \ldots, x_n$. They are closely related to double affine Hecke algebras \cite{cherednik} and are more easily extended to other root systems than the symmetric case \cite{cherednik-nsym}. 

In \cite{hhl-nsym}, Haglund, Haiman, and Loehr obtain a combinatorial formula for $\mathcal{E}_{\gamma}(x;q,t)$ that is very similar to their formula for $J_{\mu}(x;q,t)$. In fact, it is so similar that our main result goes through in that setting, albeit with a new function replacing the chromatic quasisymmetric function. 

\begin{defn}
Given a graph $G$ with its vertices numbered $1,2,\ldots,n$ and a positive integer $r \leq n$, we define the chromatic non-symmetric function to be 
\begin{align}
\mathcal{\chromatic}_{G,r}(x;t) &= \sum_{\kappa : V(G) \to \{1,2,\ldots,r\}} x^{\kappa} t^{\asc(\kappa)}
\end{align}
where the sum is over proper colorings $\kappa$ such that $\kappa(v) = n - v + 1$ if $n-r+1 \leq v \leq n$.
\end{defn}

In \cite{hhl-nsym}, the authors define new notations of arms, legs, and attacking pairs for $\gamma \in \mathbb{N}^n$. We refer the reader to \cite{hhl-nsym} for this notation. We define the attacking and augmented attacking graphs $G_{\gamma}$ and $G_{\gamma}^{+}$ analogously for these new notions of attacking and descents for $\gamma \in \mathbb{N}^n$. Our main theorem goes through for these new definitions.

\begin{prop}
\label{prop:chromatic-nsym}
\begin{align}
\mathcal{E}_{\gamma}(x;q,t) &= t^{-\sum_{u \in \gamma} \arm_{\gamma}(u)} \sum_{G_{\gamma} \subseteq H \subseteq G_{\gamma}^{+}} \mathcal{\chromatic}_{H}(x;t) \\
&\times \prod_{\{u,v\} \in H \setminus G_{\gamma}} - \left(1 - q^{\leg_{\gamma}(u) + 1} t^{\arm_{\gamma}(u)}\right) \nonumber \\
&\times \prod_{\{u,v\} \in G_{\gamma}^{+} \setminus H} \left(1 - q^{\leg_{\gamma}(u)+1}t^{\arm_{\gamma}(u)+1}\right) . \nonumber
\end{align}
\end{prop}

There are a number of ways that this setting could be explored in future work. For example, it seems like the chromatic non-symmetric functions appearing in Proposition \ref{prop:chromatic-nsym} expand positively into key polynomials, also known as Demazure characters \cite{reiner-shimozono}. This implies that the chromatic non-symmetric functions may also appear as a character. One could also try to understand the coefficients of the expansion of chromatic non-symmetric functions into key polynomials and then use these coefficients along with Proposition \ref{prop:chromatic-nsym} to obtain an expansion of $\mathcal{E}_{\gamma}(x;q,t)$ into key polynomials; such an expansion would be unique because key polynomials are linearly independent.

\section*{Acknowledgements}
The authors would like to thank Per Alexandersson, Emily Sergel Leven, Greta Panova, and George Wang for many helpful discussions and suggestions. The authors are also indebted to Erik Carlsson and Anton Mellit, since equation (11) in \cite{carlsson-mellit} inspired this entire project.

\bibliographystyle{alpha}
\bibliography{statistics}

\end{document}